\documentclass[11pt]{article}
\usepackage[a4paper,left=3cm, right=3cm,top=3.5cm, bottom=4.5cm]{geometry}
\usepackage{amsmath,amssymb,amsfonts}
\usepackage{epsfig}
\usepackage{tikz}
\usepackage{graphics}
\usepackage{bbm}
\usepackage{bm, bbm}
\usepackage{algorithm}
\usepackage{algpseudocode}
\usepackage{todonotes}
\usepackage{complexity}
\usepackage[thmmarks]{ntheorem}
\usepackage[]{hyperref}
 \hypersetup{colorlinks=true, citecolor=black,%
                 filecolor=black,%
                 linkcolor=black,%
                 urlcolor=black,  pdftex}

\usetikzlibrary[graphs, arrows, backgrounds, intersections, positioning, fit, petri, calc, shapes, decorations.pathmorphing]
\usepgflibrary{patterns}

\definecolor{darkred}{RGB}{105,0,0}
\usepackage[retainorgcmds]{IEEEtrantools}

\makeatletter
\newtheoremstyle{prime}%
 {\item[\hskip\labelsep \theorem@headerfont ##1\ \theorem@separator]}%
{\item[\hskip\labelsep \theorem@headerfont ##1\ ##3' \theorem@separator]}
\makeatother

\makeatletter
\newtheoremstyle{proofof}
{\item[\hskip\labelsep \theorem@headerfont ##1\ \theorem@separator]}%
{\item[\hskip\labelsep \theorem@headerfont ##1\ ##3\theorem@separator]}
\makeatother

\newtheorem{theorem}{Theorem}

\newtheorem{proposition}[theorem]{Proposition}

\newtheorem{corollary}[theorem]{Corollary}

\newtheorem{claim}{Claim}

\theoremstyle{prime}

\def \QD1 {\hfill $\spadesuit$}

\newcommand{\case}[2]{\noindent {\bf Case #1\/:} {\it #2}}

\numberwithin{equation}{section}

\theoremstyle {nonumberplain}
\theoremseparator{:}
\theorembodyfont{\normalfont}
\theoremsymbol{\rule{1ex}{1ex}}
\newtheorem{proof}{Proof}

\theoremstyle{proofof}

\theoremsymbol{$\square$}
\newtheorem{proof2}{Proof}

\begin{document}
\title{\bf DP-Degree Colorable Hypergraphs}

\author{{{
Thomas Schweser}
\thanks{
Technische Universit\"at Ilmenau, Inst. of Math., PF 100565, D-98684 Ilmenau, Germany. E-mail
address: thomas.schweser@tu-ilmenau.de}}
}

\date{}
\maketitle

\begin{abstract}
In order to solve a question on list coloring of planar graphs, Dvo\v{r}\'{a}k and Postle introduced the concept of so called DP-coloring, thereby extending the concept of list-coloring. DP-coloring was anaylized in detail by Bernshteyn, Kostochka, and Pron for graphs and multigraphs; they characterized DP-degree colorable multigraphs and deduced a Brooks' type result from this. The characterization of the corresponding 'bad' covers was later given by Kim and Ozeki. In this paper, the concept of DP-colorings is extended to hypergraphs having multiple (hyper-)edges. We characterize the DP-degree colorable hypergraphs and, furthermore, the corresponding 'bad' covers. This gives a Brooks' type result for the DP-chromatic number of a hypergraph. In the last part, we examine DP-critical graphs and establish some basic facts on their structure as well as a Gallai-type bound on the minimum number of edges.
\end{abstract}



\section{Introduction}
\subsection{Hypergraph Basics}
A \textbf{hypergraph} is a triple $G=(V,E,i)$, whereas $V$ and $E$ are two finite sets and $i : E \to 2^V$ is a function with $|i(e)| \geq 2$ for $e \in E$ (i.e., no loops are allowed). Then, $V(G)=V$ is the \textbf{vertex set} of $G$; its elements are the \textbf{vertices} of $G$. Furthermore, $E(G)=E$ is the \textbf{edge set} of $H$; its elements are the \textbf{edges} of $H$. Lastly, the mapping $i_G=i$ is the \textbf{incidence function} and $i_G(e)$ is the set of vertices that are $\textbf{incident}$ to the edge $e$ in $G$. Two vertices $u \neq v$ from $G$ are \textbf{adjacent} if there is an edge $e \in E(H)$ with $\{u,v\} \subseteq i_G(e)$. The \textbf{empty hypergraph} is the hypergraph $G$ with $V(G)=E(G)=\varnothing$; we denote it by $G= \varnothing$. 

Let $G$ be an arbitrary hypergraph. Then, $|V(G)|=|G|$ is called \textbf{order} of $G$. An edge $e$ of $G$ is a \textbf{hyperedge} if $|i_G(e)| \geq 3$ and an \textbf{ordinary edge}, otherwise. Thus, a \textbf{graph} is a hypergraph that contains only ordinary edges. Two edges $e \neq e'$ are \textbf{parallel}, if $i_G(e)=i_G(e')$. A \textbf{simple} hypergraph is a hypergraph without parallel edges. If $G$ is a hypergraph such that there exists an edge $e \in E(G)$ with $V(G)=i_G(e)$ and $E(G)=\{e\}$, we will brievly write $G=<e>$. 

A hypergraph $G$ is a \textbf{subhypergraph} of $G$, written $G' \subseteq G$, if $V(G') \subseteq V(G), E(G') \subseteq E(G)$, and $i_{G'}=i_G|_{E(G')}$. Furthermore, $G'$ is a proper subhypergraph of $G$ if $G' \subseteq G$ and $G' \neq G$. For two subhypergraphs $G_1$ and $G_2$ of $G$, we define the \textbf{union} $G_1 \cup G_2$ and the \textbf{intersection} $G_1 \cap G_2$ as usual. Another important operation for the class of hypergraphs that will be needed in this paper is the so called \textbf{merging} operation. Given two \textbf{disjoint} hypergraphs $G_1$ and $G_2$, that is $V(G_1) \cap V(G_2)= E(G_1) \cap E(G_2) = \varnothing$, two vertices $v^j \in V(G^j) (j \in \{1,2\})$, and a vertex $v^*$ that is not contained in $V(G_1) \cup V(G_2)$, we define a new hypergraph $G$ as follows. Let $V(G)=((V(G^1) \cup V(G^2)) \setminus \{v^1,v^2\}) \cup \{v^*\}$, $E(G)=E(G^1) \cup E(G^2)$, and 
$$i_G(e)=
\begin{cases}
i_{G^j}(e) & \text{if } e \in E(G^j), v^j \not \in i_{G^j}(e) \;(j \in \{1,2\}),\\
(i_{G^j}(e) \setminus \{v^j\}) \cup \{v^*\} & \text{if  } e \in E(G^j), v^j \in i_{G^j}(e)\;(j \in \{1,2\}).
\end{cases}$$
In this case, we say that $G$ is obtained from $G^1$ and $G^2$ by merging $v^1$ and $v^2$ to $v^*$.

 Let $G$ be a hypergraph and let $X \subseteq V(G)$ be a vertex set. Then, $G[X]$ denotes the subhypergraph of $G$ \textbf{induced by} $X$, that is, $V(G[X])=X$, $E(G[X])=\{e \in E(G) ~|~ i_G(e) \subseteq X\}$, and $i_{G[X]}=i_G|_{E(G[X])}$. Moreover, let $G-X = G[V(G) \setminus X]$. The set $X$ is called \textbf{independent} in $G$ if $E(G[X])=\varnothing$.  Finally, given a vertex $v \in V(G)$, let $G \div v$ be the hypergraph with $V(G \div v)=V(G) \setminus \{v\}$, $E(G \div v)=\{e \in E(G) ~|~ |i_G(e) \setminus \{v\}| \geq 2\}$ and $i_{G \div v}(e) = i_G(e) \setminus \{v\}$ for all $e \in E(G \div v)$. We say that $G \div v$ results from $G$ by \textbf{shrinking} $G$ at $v$. Note that if $G$ is a graph, then $G \div v = G - \{v\}$.
 
For a hypergraph $G$ and a vertex $v$ from $V(G)$, let
$$E_G(v)=\{e \in E(H) ~|~ v \in i_G(e) \}.$$
Then, $d_G(v)=|E_G(v)|$ is the \textbf{degree} of $v$ in $H$. As usual, we call $\delta(G)=\min_{v \in V(G)}d_G(v)$ the \textbf{minimum degree} of $G$ and $\Delta(G)=\max_{v \in V(G)}d_G(v)$ the \textbf{maximum degree} of $G$. If $G$ is empty, we set $\delta(G)=\Delta(G)=0$. A hypergraph $G$ is $r$\textbf{-regular} or, briefly, \textbf{regular} if each vertex in $G$ has degree $r$.
 
Let $u,v$ be two distinct vertices of a hypergraph $G$. Then, $E_G(u,v)=E(G[\{u,v\}])$ is the set of ordinary edges that are incident to $u$ as well as $v$, and $\mu_G(u,v)=|E_G(u,v)|$ is the \textbf{multiplicity} of $u$ and $v$. Note that if $u$ and $v$ are distinct vertices from $G$, then it clearly holds
\begin{align} \label{eq_Gdivv}
d_{G \div v}(u)=d_{G}(u) - \mu_G(u,v).
\end{align} 
The \textbf{ordinary neighborhood} of a vertex $v$ in a hypergraph $G$ is the set of all vertices $u \in V(G)$ such that there is an edge $e$ with $i_G(e)=\{u,v\}$, we denote it by $N_G(v)$.

A \textbf{hypermatching} of $G$, or, briefly, \textbf{matching} of $G$ is an edge set $M \subseteq E(G)$ such that each vertex $v \in V(G)$ is contained in at most one edge from the set $M$. A \textbf{perfect (hyper-)matching} is a matching $M$ such that for each $v \in V(G)$ there is a (hyper-)edge $e \in M$ with $v \in i_G(e)$. A matching $M$ is called \textbf{empty} if $M= \varnothing$. 

A \textbf{hyperpath} of a non-empty hypergraph $G$ is a sequence $(v_1, e_1, v_2, e_2, \ldots, v_q, e_q, v_{q+1})$ of distinct vertices
$v_1, v_2, \ldots, v_{q+1}$ of $G$ and distinct edges
$e_1, e_2, \ldots, e_q$ of $G$ such that
$\{v_i,v_{i+1}\}\subseteq i_G(e_i)$ for $1 \leq i \leq q$ and $q \geq 0$. The hypergraph $G$ is \textbf{connected} if there is a hyperpath in $G$ between any two of its vertices. A (connected) \textbf{component} of $G$ is a maximal connected subhypergraph of $G$. A \textbf{separating vertex} of $G$ is a vertex $v \in V(G)$ such that $G$ is the union of two induced subhypergraph $G_1$ and $G_2$ with $V(G_1)\cap V(G_2)=\{v\}$ and $|G_i|\geq 2$ for $i\in \{1,2\}$. Note that $v$ is a separating vertex of $G$ if and only if $G \div v$ has more components than $G$. Regarding edges, an edge $e$ is a \textbf{bridge} of a hypergraph $G$, if $G-e$ has $|i_G(e)|-1$ more components than $G$. Finally, a \textbf{block} of $G$ is a maximal connected subhypergraph of $G$ that has no separating vertex. Thus, every block of $G$ is a connected induced subhypergraph of $G$. It is easy to see that two blocks of $G$ have at most one vertex in common, and that a vertex $v$ is a separating vertex of $G$ if and only if it is contained in more than one block. By $\mathcal{B}(G)$ we denote the set of all blocks of $G$. If $G$ does not contain any separating vertex, that is, $\mathcal{B}(G)=\{G\}$, we will also say that $G$ is a \textbf{block}.

In this paper we need some definitions only for graphs. As usual, by $C_n$ we denote the cycle with $n$ vertices and by $K_n$ we denote the complete graph on $n$ vertices. Moreover, for $n,t \geq 1$, by $K_{(n,t)}$ we denote the complete $n$-partite graph all of whose partite sets have $t$ vertices. In particular, $K_{(2,t)}$ is the complete bipartite graph $K_{t,t}$. If $G$ is a simple graph, let $tG$ denote the graph that results from $G$ by replacing each edge $e$ with $t$ parallel edges. In particular, $1G=G$. By $A(G)$ we denote the set of all two-subsets $\{u,v\} \in V(G)$ such that there in $G$ there is an edge $e$ with $i_G(e)=\{u,v\}$.

\subsection{Hypergraph Colorings}
A \textbf{coloring} of a hypergraph $G$ with \textbf{color set} $C$ is a mapping $\varphi: V(G) \to C$ such that for each edge $e$ there are vertices $u,v \in i_G(e)$ with $\varphi(u) \neq \varphi(v)$. If $C=\{1,2,\ldots,k\}$ and if $G$ admits a coloring with color set $C$, we say that $G$ is $k$-colorable. The smallest $k \geq 0$ such that $G$ is $k$-colorable is called the \textbf{chromatic number} $\chi(G)$ of $G$. This coloring concept was introduced by Erd\H os and Hajnal\cite{ErdHaj66} in the 1960s.

Note that if $G$ is a graph, this coloring concept coincides with the usual coloring concept for graphs. In particular, it is possible to extend various well known theorems in the topic of graph colorings to hypergraph colorings. For example, Brooks' well known theorem (see \cite{brooks}) that a connected graph $G$ satisfies $\chi(G) \leq \Delta(G) + 1$ and equality holds if and only if $G$ is a complete graph or an odd cycle, was extended to hypergraphs by Jones \cite{Jones}. He showed that if $G$ is a connected hypergraph, then $\chi(G) \leq \Delta(G) + 1$ and equality holds if and only if $G$ is a complete graph, an odd cycle, or $G=<e>$ for some edge $e$.

 A more generalized coloring concept is the so called \textbf{list-coloring concept}. Let again  $C$ be a color set and let $G$ be a hypergraph. A mapping $L:V(G) \to 2^C$ is called \textbf{list-assignment}. An $L$\textbf{-coloring} of $G$ is a coloring $\varphi$ of $G$ with color set $C$ such that $\varphi(v) \in L(v)$ for all $v \in V(G)$. The hypergraph $G$ is said to be $k$\textbf{-list-colorable} if $G$ admits an $L$-coloring for any list-assignment $L$ satisfying $|L(v)| \geq k$ for all $v \in V(G)$. The \textbf{list-chromatic number} or \textbf{choice number} $\chi_\ell(G)$ is the least integer $k$ such that $G$ is $k$-list-colorable. For graphs, list-colorings were introduced by Erd\H os, Rubin, and Taylor \cite{ErdRubTay79} and, independently, by Vizing \cite{Vizing}. They proved a Brooks type theorem, which was later extended to hypergraphs by Kostochka, Stiebitz, and Wirth \cite{KoStiWi96}. In this paper, we will generalize their theorem to DP-colorings of hypergraphs. 
 
\section{DP-Colorings of Hypergraphs}
This paper deals with \textbf{DP-colorings} of hypergraphs. For graphs, this concept was introduced by Dvo\v{r}\'{a}k and Postle \cite{DvoPo15}; they called it \textbf{correspondence coloring}. The orginial idea was taken from Plesnevi\v{c} and Vizing \cite{PleVi65} who showed how to transform the problem of finding a $k$-coloring of a graph to the problem of finding an independent vertex set of size $|V(G)|$ in the Cartesian product $G$\scalebox{0.8}{$\square$}$K_k$. Later, a lot of work on the topic of DP-colorings was done by Bernshteyn, Kostochka et al. (see \cite{BeKo17}, \cite{BeKoPro17}, \cite{BeKoZhu17}), who were the first to use the term DP-colorings. We will use an equivalent but slightly modified definition.

\subsection{The DP-Chromatic Number}
Let $G$ be a hypergraph. A \textbf{cover} of $G$ is a pair $(X,H)$ consisting of a map $X$ and a hypergraph $H$ such that the following conditions are fulfilled:

\begin{itemize}
\item[\textbf{(C1)}] $X: V(G) \to 2^{V(H)}$ is a function that assigns each vertex $v \in V(G)$ a vertex set $X_v = X(v) \subseteq V(H)$ such that the sets $X_v$ with $v \in V(G)$ are pairwise disjoint.

\item[\textbf{(C2)}] $H$ is a hypergraph with $V(H) = \bigcup_{v \in V(G)}X_v$ such that $X_v$ is an independent set of $H$, 
and for each edge $e \in E(G)$ there is a possibly empty (hyper-)matching $M_e$ in $H[\bigcup_{v \in i_G(e)} X_v]$ with $|i_H(\tilde{e}) \cap X_v|=1$ for all $v \in i_G(e)$ and for all $\tilde{e} \in M_e$. Moreover, $E(H)=\bigcup_{e \in E(G)}M_e$.
\end{itemize}

Now let $(X,H)$ be a cover of $G$. A vertex set $T \subseteq V(H)$ is a \textbf{transversal} of $(X,H)$ if $|T \cap X_v|=1$ for each vertex $v \in V(G)$. An \textbf{independent transversal} of $(X,H)$ is a transversal of $(X,H)$, which is an independent set of $H$. An independent transversal of $(X,H)$ is also called an $(X,H)$\textbf{-coloring} of $G$; the vertices of $H$ are called \textbf{colors}. We say that $G$ is $(X,H)$-colorable if $G$ admits an $(X,H)$-coloring. Let $f: V(G) \to \mathbb{N}_0$ be a function. Then, $G$ is said to be \textbf{DP-}$f$\textbf{-colorable} if $G$ is $(X,H)$-colorable for any cover $(X,H)$ of $G$ satisfying $|X_v| \geq f(v)$ for all $v \in V(G)$. When $f(v)=k$ for all $v \in V(G)$, the term becomes \textbf{DP-}$k$\textbf{-colorable}. The \textbf{DP-chromatic number} $\chi_{\text{DP}}(G)$ is the least integer $k \geq 0$ such that $G$ is DP-$k$-colorable. Recently, Bernshteyn and Kostochka \cite{BeKo18} also introduced the DP-chromatic number of a hypergraph in an equivalent but slightly different way.

An especially interesting fact about DP-colorings is that one can reduce the list-coloring problem to DP-colorings. To see this, let $G$ be a hypergraph and let $L$ be a list-assignment for $G$. Let $(X,H)$ be a cover of $G$ as follows:

\begin{itemize}
\item For $v \in V(G)$, let $X_v=\{(v,x)~|~ x \in L(v)\}$ and let $V(H)=\bigcup_{v \in V(G)}X_v$.
\item For any set $S=\{(v_1,x_1),(v_2,x_2),\ldots,(v_\ell,x_\ell)\}$ of vertices from $H$, there is an edge $e' \in E(H)$ with $i_H(e')=S$ if and only if in $G$ there is an edge $e$ with $i_G(e)=\{v_1,v_2,\ldots,v_\ell\}$ and if $x_1=x_2=\ldots=x_\ell$.
\end{itemize}

It is easy to check that $(X,H)$ is indeed a cover of $G$. Furthermore, if $\varphi$ is an $L$-coloring of $G$, then $T=\{(v, \varphi(v)) ~|~ v \in V(G)\}$ clearly is an independent transversal of $H$ and so $G$ is $(X,H)$-colorable. If conversely $T$ is an independent transversal of $H$, then there is a mapping $\varphi$ from $V(G)$ into a color set such that $\varphi(v) \in L(v)$ for all $v \in V(G)$ and that $T=\{(v,\varphi(v)) ~ | ~ v \in V(G)\}$. Then, it is easy to see that $\varphi$ is an $L$-coloring of $G$. Thus, $G$ is $L$-colorable if and only if $G$ is $(X,H)$-colorable. Furthermore, we clearly have $|X_v|=|L(v)|$ for all $v \in V(G)$. Hence, if $k \geq 0$ is an integer, then $G$ is $k$-list-colorable if $G$ is DP-$k$-colorable and, in particular, $\chi_\ell(G) \leq \chi_{\text{DP}}(G)$.

In order to obtain an upper bound for the DP-chromatic number we use a sequential coloring algorithm. 

\begin{algorithm}
\caption{Sequential coloring algorithm}
\begin{algorithmic}[1]
\State Input: hypergraph $G$ and cover $(X,H)$ of $G$.
\State Choose an arbitrary vertex order $(v_1,v_2,\ldots,v_n)$ of $G$.
\State Let $T=\varnothing$.
\ForAll{$i=1,2,\ldots,n$}
\State{Choose a vertex (color) $x_i$ from $X_{v_i}$ such that $E(H[T\cup\{x_i\}])=\varnothing$}.
\State{Let $T=T \cup \{x_i\}$}.
\EndFor
\State Return: Independent transversal $T$.
\end{algorithmic}
\end{algorithm}

Clearly, if $|X_{v_i}| \geq d_{G[\{v_1,v_2,\ldots,v_i\}]}(v_i)+ 1$ for all $i \in \{1,2,\ldots,n\}$, in step 5 there is always a possible choice for $x_i$ and, thus, the algorithm terminates with an $(X,H)$-coloring of $G$. This is due to the fact that for each edge $e \in E(G)$ with $v_i \in i_H(e)$ and for any set of fixed colors $$\{x_k ~|~ v_k \in i_H(e), k=1,2,\ldots,i-1\},$$
 at most one color from $X_{v_i}$ is forbidden. Hence, in $X_{v_i}$, at most $d_{G[v_1,v_2,\ldots,v_i]}(v_i)$ vertices are forbidden. As a consequence, if $f(v) \geq d_G(v) + 1$ for all $v \in V(G)$, then $G$ is DP-$f$-colorable. As usual, the \textbf{coloring number} $\text{col}(G)$ of a hypergraph $G$ is the least integer $k$ such that each non-empty subhypergraph contains a vertex of degree at most $k$. Therefore, as a consequence of the above sequential coloring algorithm, we have $\chi_{\text{DP}}(G)\leq \text{col}(G)$. Summarizing, we obtain
\begin{align}\label{eq_chi-inequalities}
\chi(G) \leq \chi_\ell(G) \leq \chi_{\text{DP}}(G) \leq \text{col}(G) \leq \Delta(G) + 1.
\end{align} 

Our aim is to characterize the hypergraphs $G$ for which $\chi_{\text{DP}}(G) = \Delta(G) + 1$ holds. Clearly, if $G$ is an odd cycle, we have $\chi(G)=\Delta(G)+1=3$ and, thus, equality holds. To see that $\chi_{\text{DP}}(G)=3$ holds for even cycles, as well, we construct an appropriate cover of $G$. Assume that $V(G)=\{1,2,\ldots,n\}$ with $n \geq 2$ even and $E(G)=\{uv ~|~ u,v \in V(G) \text{ and } u-v \equiv 1 \text{(mod } n \text{)}\}$. Let $(X,H)$ be the cover of $G$ with $X_v=\{v\} \times \{1,2\}$ for all $v \in V(G)$ and $E(H)=\{(u,i)(v,j) ~|~ |u-v| = 1 \text{ and } i=j; \text{ or } \{u,v\}=\{1,n\} \text{ and } i-j \equiv 1 \text{ (mod } 1 \text{)}\}$. Then, $(X,H)$ is a cover of $G$ with $|X_v|=2$ for all $v \in V(G)$. Moreover, $H=C_{2n}$ and $(X,H)$ has no independent transversal. As emphasized in \cite{BeKoPro17}, the fact that $\chi_{\text{DP}}(C_n)=3$ for all $n \geq 2$ and not only for odd $n \geq 3$ marks an important difference between the DP-chromatic number and the list-chromatic number.

\subsection{DP-Degree Colorable Hypergraphs}

We say that a hypergraph $G$ is \textbf{DP-degree colorable} if $G$ is $(X,H)$-colorable whenever $(X,H)$ is a cover of $G$ such that $|X_v| \geq d_G(v)$ for all $v \in V(G)$. Regarding graphs, Bernshteyn, Kostochka, and Pron~\cite{BeKoPro17} proved that a  connected graph $G$ is not DP-degree colorable if and only if each block of $G$ is a $tK_n$ or a $tC_n$ for some integers $t,n \geq 1$. Of course, when dealing with DP-coloring, is it not only of interest to characterize the non DP-degree colorable graphs, but also the corresponding 'bad' covers. This was done by Kim and Ozeki \cite{KiOz17} (see Theorem~\ref{theorem_Kim-Ozeki}). The aim of this section is to give the corresponding characterizations for DP-degree-colorable hypergraphs.

A \textbf{feasible configuration} is a triple $(G,X,H)$ consisting of a connected hypergraph $G$ and a cover $(X,H)$ of $G$. A feasible configuration is said to be \textbf{degree-feasible} if $|X_v| \geq d_G(v)$ for each vertex $v \in V(G)$. Furthermore, $(G,X,H)$ is \textbf{colorable} if $G$ is $(X,H)$-colorable, otherwise it is called \textbf{uncolorable}.

The next proposition lists some basic properties of feasible configurations; the proofs are straightforward and left to the reader.

\begin{proposition} \label{prop_feas-config}
Let $(G,X,H)$ be a feasible configuration. Then, the following statements hold. 
\begin{itemize}
\item[\upshape (a)] For distinct vertices $u,v$ of $G$, the hypergraph $H[X_u \cup X_v]$ is a bipartite graph with parts $X_u$ and $X_v$ whose maximum degree is at most $\mu_G(u,v)$. Furthermore, for every vertex $v \in V(G)$ and every vertex $x \in X_v$, we have $d_H(x) \leq d_G(v)$.
\item[\upshape (b)] Let $H'$ be a spanning subhypergraph of $H$. Then, $(G,X,H')$ is a feasible configuration. If $(G,X,H)$ is colorable, then $(G,X,H')$ is colorable, too. Furthermore, $(G,X,H)$ is degree-feasible if and only if $(G,X,H')$ is degree feasible.
\end{itemize}
\end{proposition}

The above proposition leads to the following concept. We say that a feasible configuration $(G,X,H)$ is \textbf{minimal uncolorable} if $(G,X,H)$ is uncolorable, but $(G,X,H-e)$ is colorable for each $e \in E(H)$. As usual, $H-e$ denotes the hypergraph obtained from $H$ by deleting the edge $e$. Clearly, if $|G| \geq 2$ and if $\tilde{H}$ is the edgeless spanning hypergraph of $H$, then $(G,X,\tilde{H})$ is colorable. Thus, it follows from the above Proposition that if $(G,X,H)$ is an uncolorable feasible configuration, then there is a spanning subhypergraph $H'$ of $H$ such that $(G,X,H')$ is a minimal uncolorable feasible configuration. Furthermore, if $(G,X,H)$ is a minimal uncolorable feasible configuration, then $H$ clearly is a simple hypergraph.

In order to characterize the class of minimal uncolorable degree-feasible configurations, we firstly need to introduce three basic types of degree-feasible configurations.

We say that $(G,X,H)$ is a \textbf{K-configuration} if $G=tK_n$ for some integers $t,n \geq 1$ and if $(X,H)$ is a cover of $G$ such that for every vertex $v \in V(G)$, there is a partition $(X_v^1,X_v^2,\ldots,X_v^{n-1})$ of $X_v$ satisfying the following conditions:
\begin{itemize}
\item For every $i \in \{1,2,\ldots,n-1\}$, the graph $H^i=H[\bigcup_{v \in V(G)} X_v^i]$ is a $K_{(n,t)}$ whose partite sets are the sets $X_v^i$ with $v \in V(G)$, and
\item $H=H^1 \cup H^2 \cup \ldots \cup H^{n-1}$.
\end{itemize}
It is an easy exercise to check that each $K$-configuration is a minimal uncolorable degree-feasible configuration. Note that for $n=1$, we have $G=K_1$, $X=\varnothing$, and $H=\varnothing$.

Next we define the so called C-configurations. We say that $(G,X,H)$ is an \textbf{odd C-configuration} if $G=tC_n$ for some integers $t \geq 1$ and $n \geq 5$ odd and if $(X,H)$ is a cover of $G$ such that for every vertex $v \in V(G)$, there is a partition $(X_v^1,X_v^2)$ of $X_v$ satisfying the following conditions:
\begin{itemize}
\item For every $i \in \{1,2\}$ and for every set $\{u,v\} \in A(G)$, the graph $H^i_{\{u,v\}}=H[X_u^i \cup X_v^i]$ is a $K_{t,t}$ whose partite sets are $X_u^i$ and $X_v^i$, and
\item $H$ is the union of all graphs $H^i_{\{u,v\}}$ with $i \in \{1,2\}$ and $\{u,v\} \in A(G)$.
\end{itemize}
It is easy to verify that any odd C-configuration is a minimal uncolorable degree-feasible configuration.

We call $(G,X,H)$ an \textbf{even C-configuration} if $G=tC_n$ for some integers $t \geq 1, n \geq 4$ even and if $(X,H)$ is a cover of $G$ such that for every vertex $v \in V(G)$, there is a partition $(X_v^1,X_v^2)$ of $X_v$ and a set $\{w,w'\} \in A(G)$ satisfying the following conditions:
\begin{itemize}
\item For every $i \in \{1,2\}$ and for every set $\{v,w\} \in A(G)$ different from $\{w,w'\}$, the graph $H^i_{\{u,v\}}=H[X_u^i \cup X_v^i]$ is a $K_{t,t}$ whose partite sets are $X_u^i$ and $X_v^i$,
\item $H^1_{\{w,w'\}}=H [X_w^1 \cup X_{w'}^2]$ is a $K_{t,t}$ whose partite sets are $X_w^1$ and $X_{w'}^2$,
\item $H^2_{\{w,w'\}}=H[X_w^2 \cup X_{w'}^1]$ is a $K_{t,t}$ whose partite sets are $X_w^2$ and $X_{w'}^1$, and
\item $H$ is the union of all graphs $H^i_{\{u,v\}}$ with $i \in \{1,2\}$ and $\{u,v\} \in A(G)$.
\end{itemize}
Again, it is easy to check that any even C-configuration is a minimal uncolorable degree-feasible configuration. By a \textbf{C-configuration} we either mean an even or an odd C-configuration. 

Finally, we say that $(G,X,H)$ is an \textbf{E-configuration} if $G=<e>$ for some $e \in E(H)$, if $|X_v| = 1$ for each $v \in V(G)$ and if $H \cong G$. Clearly, each E-configuration is a minimal uncolorable degree-feasible configuration.

We will show that we can construct any minimal uncolorable degree-feasible configuration from these three basic configurations using the following operation. Let $(G^1,X^1,H^1)$ and $(G^2,X^2,H^2)$ be two feasible configurations, which are \textbf{disjoint}, that is, $V(G^1) \cap V(G^2) = \varnothing$ and $V(H^1) \cap V(H^2)= \varnothing$. Furthermore, let $G$ be the hypergraph obtained from $G^1$ and $G^2$ by merging two vertices $v^1 \in V(G^1)$ and $v^2 \in V(G^2)$ to a new vertex $v^*$. Finally, let $H=H^1 \cup H^2$ and let $X:V(G) \to 2^{V(H)}$ be the mapping such that 
$$X_v=\begin{cases}
X^1_{v^1} \cup X^2_{v^2} & \text{if } v=v^*,\\
X_v^i & \text{if } v \in V(G^i) \setminus \{v^i\} \text{ and } i \in \{1,2\}
\end{cases}$$
for $v \in V(G)$. Then, $(G,X,H)$ is a feasible configuration and we say that $(G,X,H)$ is obtained from $(G^1,X^1,H^1)$ and $(G^2,X^2,H^2)$ by \textbf{merging} $v^1$ and $v^2$ to $v^*$. Since $d_G(v^*)=d_{G^1}(v^1) + d_{G^2}(v^2)$, it follows that $(G,X,H)$ is degree-feasible if both $(G^1,X^1,H^1)$ and $(G^2,X^2,H^2)$ are degree-feasible.

Now we define the class of \textbf{constructible configurations} as the smallest class of feasible configurations that contains each K-configuration, each C-configuration and each E-configuration and that is closed under the merging operation. Thus, if $(G,X,H)$ is a constructible configuration, then each block of $G$ is a $tK_n$ for $t \geq 1, n \geq 1$, a $tC_n$ for $t \geq 1, n \geq 3$, or of the form $<e>$ for some edge $e$. We call a block $B \in \mathcal{B}(G)$ a \textbf{DP-brick} if $B=tK_n$ for some $t \geq 1, n \geq 1$ or if $B=tC_n$ for some $t \geq 1, n \geq 3$. Moreover, we say that $B \in \mathcal{B}(G)$ is a \textbf{DP-hyperbrick}, if $B$ is either a DP-brick or of the form $<e>$ for some edge $e$.

Now we are able to state our main result.

\begin{theorem} \label{theorem_main-result}
Let $(G,X,H)$ be a degree-feasible configuration. Then, $(G,X,H)$ is minimal uncolorable if and only if $(G,X,H)$ is constructible.
\end{theorem}

Before we can prove Theorem~\ref{theorem_main-result}, we need various propositions. The next one describes the block-configurations of constructible configurations, the proof can be done by induction on the number of blocks and is left to the reader.

\begin{proposition} \label{prop_block-structure}
Let $(G,X,H)$ be a constructible configuration. Then, for each block $B \in \mathcal{B}(G)$ there is a uniquely determined cover $(X^B,H^B)$ of $B$ such that the following statements hold:
\begin{itemize}
\item[\upshape (a)] For each block $B \in \mathcal{B}(G)$, the triple $(B,X^B,H^B)$ is a {\upshape K}-configuration, a {\upshape C}-configuration, or an {\upshape E}-configuration.
\item[\upshape (b)] The hypergraphs $H^B$ with $B \in \mathcal{B}(G)$ are pairwise disjoint and $H = \cup_{B \in \mathcal{B}(G)} H^B$.
\item[\upshape (c)] For every vertex $v \in V(G)$ it holds $X_v = \bigcup_{B \in \mathcal{B}(G), v \in V(B)}X_v^B$.
\end{itemize}
\end{proposition}

The next proposition is key in order to obtain our main result, it describes a reduction method that allows us to use induction on the number of vertices of $G$. Similar propositions to the next two propositions were proven by Bernshteyn, Kostochka, and Pron for graphs in \cite{BeKoPro17}.

\begin{proposition} \label{proposition_reduction}
Let $(G,X,H)$ be a feasible configuration with $|G| \geq 2$, let $v$ be a non-separating vertex of $G$, and let $x \in X_v$ be a color. We define a cover of the hypergraph $G'=G \div v$ as follows. For $u \in V(G')$ let
$$X'_u = X_u \setminus N_H(x)$$
and let $H'$ be the hypergraph with $V(H') = \bigcup_{u \in V(G')} X_u'$,
$$E(H') = \{ e ~|~ e \in E(H), |i_H(e) \setminus \{x\} | \geq 2, \text{ and } (i_{H}(e) \setminus \{x\}) \subseteq V(H')\},$$
and
$$i_{H'}(e)=i_H(e) \setminus \{x\}$$
for all $e \in E(H')$.

Then, $(G',X',H')$ is a feasible configuration, and in what follows we briefly write $(G',X',H')=(G,H,X)/(v,x)$. Moreover, the following statements hold:
\begin{itemize}
\item[\upshape (a)] If $(G,X,H)$ is degree-feasible, then $(G',X',H')$ is degree-feasible, too.
\item[\upshape (b)] If $(G,X,H)$ is uncolorable, then $(G',X',H')$ is uncolorable, too.
\end{itemize}
\end{proposition}

\begin{proof}
Clearly, $(X',H')$ is a cover of $G'$ and, hence, $(G',X',H')$ is a feasible configuration. Moreover, for $u \in V(G')$ it holds $d_{G'}(u)=d_{G}(u) - \mu_G(u,v)$ and $|N_H(x) \cap X_u| \leq \mu_G(u,v)$ (see \eqref{eq_Gdivv} and Proposition~\ref{prop_feas-config}). Thus, we obtain
$$|X'_u| = |X_u| - |N_H(x) \cap X_u| \geq |X_u| - \mu_G(u,v).$$
As $|X_u| \geq d_G(u)$, this leads to $|X_u'| \geq d_{G'}(u)$ and $(G',H',X')$ is degree-feasible. Furthermore, if $T'$ is an independent transversal of $(X',H')$, then $T=T'\cup\{x\}$ is an independent transversal of $(X,H)$. This proves (b).
\end{proof}

Using the above introduced reduction method, we obtain the following. 

\begin{proposition}\label{prop_main-proposition}
Let $(G,X,H)$ be an uncolorable degree-feasible configuration. Then, the following statements hold:
\begin{itemize}
\item[\upshape (a)] $|X_v|=d_G(v)$ for all $v \in V(G)$.
\item[\upshape (b)] For each non-separating vertex $z$ of $G$ and each vertex $v \neq z$ of $G$, it holds $|N_H(x) \cap X_v| = \mu_G(v,z)$ for all $x \in X_z$.
\item[\upshape (c)] Every hyperedge $e$ of $G$ is a bridge of $G$ and, therefore, $<e>$ is a block of $G$. As a consequence, there are no parallel hyperedges in $G$.
\item[\upshape (d)] If $G$ is a block, then $G$ is regular, and for distinct vertices $u,v$ of $G$, the hypergraph $H[X_u \cup X_v]$ is a $\mu_G(u,v)$-regular bipartite graph whose partite sets are $X_u$ and $X_v$.
\item[\upshape (e)] For each vertex $v \in V(G)$ there is an independent set $T$ in $H$ satisfying $|T \cap X_u| = 1$ for all $u \in V(G)  \setminus \{v\}$.

\end{itemize}
\end{proposition}

\begin{proof}
We prove (a) by induction on the order of $G$. If $G$ consists of only one vertex $v$, then $X_v= \varnothing$ and $H = \varnothing$. Thus, (a) is fulfilled. Now assume $|G| \geq 2$ and choose an arbitrary vertex $v$ of $G$. As $G$ is connected, there is a non-separating vertex $z \neq v$ in $G$ and $X_z \neq \varnothing$. Let $x \in X_z$. Then, $(G',X',H')=(G,X,H)/(z,x)$ is an uncolorable degree-feasible configuration (by Proposition~\ref{proposition_reduction}). Applying the induction hypothesis then leads to $|X_v'|=d_{G'}(v)$ and we conclude

\begin{IEEEeqnarray*}{rCCCl}
d_{G'}(v) & = & |X'_v| & = & |X_v| - |N_H(x) \cap X_v| \\
&&& \geq & |X_v| - \mu_G(v,z) \geq d_G(v) - \mu_G(v,z) = d_{G'}(v).
\end{IEEEeqnarray*}
This implies $|X_v|=d_G(v)$ and $|N_H(x) \cap X_v|=\mu_G(v,z)$; thus, (a) is proven. The same argument can be applied in order to prove (b). 

For the proof of (c) assume that some hyperedge $e \in E(G)$ is not a bridge of $G$. Then, for some vertex $v \in i_G(e)$, the hypergraph $G'=(V(G),E(G) \setminus \{e\} \cup \{e - v\})$ is connected. Let $X'=X$ and let $H'$ be the hypergraph with vertex set $V(H)$ and edge set $(E(H) \setminus M_e) \cup M_{e -v}$, whereas $M_{e-v}$ denotes the restriction of $M_e$ to the vertices of $\bigcup_{u \in i_G(e) \setminus\{v\}} X_u$. Clearly, $(G',X',H')$ is a degree-feasible configuration. However, (a) implies that $|X_v'|=|X_v|=d_G(v) > d_{G'}(v)$ and so, again by (a), $(G',X',H')$ is colorable. Hence, there is an independent transversal $T'$ of $(G',X',H')$. We claim that $T'$ is also an independent transversal of $(G,X,H)$. Otherwise, by construction of $H'$ there would be an edge $\tilde{e} \in E(H)$ with $v \in i_H(\tilde{e}) \subseteq T'$. But then, $i_{H'}(\tilde{e}-v) \subseteq T'$ and so $T'$ is not an independent transversal of $H'$, a contradiction. Hence, $T'$ is an independent transversal of $(G,X,H)$ and so $(G,X,H)$ is colorable, which is impossible. This settles the case (c).

In order to prove (d), assume that $G$ is a block. If $G=<e>$ for some hyperedge $e$, then $G$ is regular and the statement clearly holds. Thus, by (c), we may assume that $G$ does not contain any hyperedge. Let $u,v$ be distinct vertices of $G$. Then, $H[X_u \cup X_v]$ is a $\mu_G(u,v)$-regular bipartite graph with parts $X_u$ and $X_v$ (by (b)). This is only possible if $|X_u|=|X_v|$. By (a), this leads to $d_G(u)=d_G(v)$ and (d) is proven.

Finally, for the proof of (e), let $v$ be an arbitrary vertex of $G$. Let $U$ be the vertex set of a component of $G-v$, let $X'$ be the restriction of $X$ to $U$, and let $H'=H[\bigcup_{u \in U} X_u]$. Then, $(G[U],X',H')$ is a degree-feasible configuration and, as $G$ is connected, it holds $|X_u| = d_G(u) > d_{G[U]}(u)$ for at least one vertex $u \in U$. Hence, there is an independent transversal $T_u$ of $(X',H')$ (by (a)). Let $T$ be the union of the transversals $T_U$ over all components $G[U]$ of $G-v$. Clearly, $T$ is an independent set of $H$ such that $|T \cap X_w|=1$ for all $w \in V(G) \setminus \{v\}$. This proves (e).
\end{proof}

Finally, we connect the concept of being minimal uncolorable with the merging operation.

\begin{proposition} \label{prop_minimal-uncolorable}
Let $(G,X,H)$ be obtained from two disjoint degree-feasible configurations $(G^1,X^1,H^1)$ and $(G^2,X^2,H^2)$ by merging $v^1 \in V(G^1)$ and $v^2 \in V(G^2)$ to a new vertex $v^*$. Then, $(G,X,H)$ is a degree-feasible configuration and the following conditions are equivalent:
\begin{itemize}
\item[\upshape (a)] Both $(G^1,X^1,H^1)$ and $(G^2,X^2,H^2)$ are minimal uncolorable.
\item[\upshape (b)] $(G,X,H)$ is minimal uncolorable.
\end{itemize}
\end{proposition}

\begin{proof}
First we show that (a) implies (b). Assume that $(G,X,H)$ is colorable. Then, there is an independent transversal $T$ of $(X,H)$, that is, an independent set of $H$ such that $|T \cap X_u|=1$ for all $u \in V(H)$. As $X_{v^*}=X_{v^1} \cup X_{v^2}$, this implies (by symmetry) that $|T \cap X_{v^1}|=1$. As a consequence, $T^1=T \cap V(H^1)$ is an independent transversal of $(X^1,H^1)$ and so $(G^1,X^1,H^1)$ is colorable, a contradiction to (a). Thus, $(G,X,H)$ is uncolorable. Let $e \in E(H)$ be an arbitrary (hyper)-edge. By the structure of $H=H^1 \cup H^2$, we may assume that $e \in E(H^1)$. Due to the fact that $(G^1,X^1,H^1)$ is minimal uncolorable, there is an independent transversal $T^1$ of the cover $(X^1,H^1)$. Since $(G^2,X^2,H^2)$ is also minimal uncolorable and as $G^2$ is connected, it follows from Proposition~\ref{prop_main-proposition}(e) that there is an independent set $T^2$ in $H^2$ satisfying $|T^2 \cap X_u^2|=1$ for all $u \in V(G^2) \setminus \{v^2\}$. However, as $H=H^1 \cup H^2$ and $H_1 \cap H_2 = \varnothing$, the set $T=T^1 \cup T^2$ is an independent transversal of $(X,H-e)$ and so $(G,X,H-e)$ is colorable. Thus, (b) holds.

In order to prove that (a) can be deduced from (b), we only need to show that $(G^1,X^1,H^1)$ is minimal uncolorable (by symmetry). First assume that $(G^1,X^1,H^1)$ is colorable, that is, $(X^1,H^1)$ has an independent transversal $T^1$. Since $(G,X,H)$ is minimal uncolorable and connected and as $H^2-X_{v^2}$ is a subhypergraph of $H$, Proposition~\ref{prop_main-proposition}(d) implies that there is an independent set $T^2$ in $H^2-X_{v^2}$ such that $|T^2 \cap X_u^2|=1$ for all $u \in V(G^2) \setminus \{v^2\}$. Then again, $T=T^1 \cup T^2$ is an independent transversal of $(X,H)$, contradicting (b). Thus, $(G^1,X^1,H^1)$ is uncolorable. Now let $e \in E(H^1)$ be an arbitrary edge. Then, as $(G,X,H)$ is minimal uncolorable, there is an independent transversal $T$ of $(X,H-e)$ and $T^1=T \cap V(H^1)$ clearly is an independent transversal of $H^1$. Consequently, $(G^1,X^1,H^1-e)$ is colorable and the proof is complete. 
\end{proof}

\subsection{Proof of Theorem~\ref{theorem_main-result}}
Before proving it, let us again state the main result of this paper.

\setcounter{theorem}{1}
\begin{theorem}
Let $(G,X,H)$ be a degree-feasible configuration. Then, $(G,X,H)$ is minimal uncolorable if and only if $(G,X,H)$ is constructible.
\end{theorem}

\begin{proof}
If $(G,X,H)$ is constructible, then $(G,X,H)$ is minimal uncolorable (by Proposition~\ref{prop_minimal-uncolorable} and as each K,C and E-configuration is a minimal uncolorable degree-feasible configuration). Let $(G,X,H)$ be a minimal uncolorable degree-feasible configuration. We prove that $(G,X,H)$ is constructable by induction on the order of $G$. Clearly, if $|G|=1$, then $X=\varnothing$, $H=\varnothing$ and $(G,X,H)$ is a K-configuration. Assume that $|G| \geq 2$. By Proposition~\ref{prop_main-proposition}(a), it holds
\begin{align}\label{eq_degree}
|X_v|=d_G(v)
\end{align}
for each vertex $v \in V(G)$. We distinguish between two cases. 

\case{1}{$G$ contains a separating vertex $v^*$.} Then, $G$ is the union of two connected induced subhypergraphs $G^1$ and $G^2$ with $V(G^1) \cap V(G^2) = \{v^*\}$ and $|G^j| < |G|$ for $j \in \{1,2\}$. For $j \in \{1,2\}$, by $\mathcal{T}^j$ we denote the set of all independent sets $T$ of $H$ such that $|T \cap X_v|=1$ for all $v \in V(G^j)$. By Proposition~\ref{prop_main-proposition}(e), both $\mathcal{T}^1$ and $\mathcal{T}^2$ are non-empty. For $j \in \{1,2\}$, let $X_j$ be the set of all vertices of $X_{v^*}$ that do not occur in any independent set from $\mathcal{T}^j$. Then, $X_{v^*}=X_1 \cup X_2$. Suppose, to the contrary, that there is a vertex $u \in X_{v^*} \setminus (X_1 \cup X_2)$. Then, $u$ is contained in two independent sets $T^j \in \mathcal{T}^j (j=1,2)$ and $T=T^1 \cup T^2$ would be an independent transversal of $(X,H)$. This is due to the fact that each hyperedge of $G$ is contained in $G^j$ for some $j \in \{1,2\}$ and that for $u \in V(G^1) \setminus \{v^*\}$ and $v \in V(G^2) \setminus \{v^*\}$ we have $\mu_G(u,v)=0$ and so $E_H(X_u,X_v)= \varnothing$ (by Proposition~\ref{prop_feas-config}(a)). Thus, $(G,X,H)$ is colorable, a contradiction. Consequently, $X_{v^*}=X_1 \cup X_2$. For $j \in \{1,2\}$, let $(X^j,H^j)$ be a cover of $G^j$ as follows. For $v \in V(G^j)$, let
$$X_v^j=\begin{cases}
X_v & \text{if } v \neq v^* \\
X_j & \text{if } v = v^*,
\end{cases}$$
and let $H^j=H[\bigcup_{v \in V(G^j)} X_v^j]$. Then, $(G^j,X^j,H^j)$ is a feasible configuration and, by definition of $X_j=X_{v^*}^j$, $(G^j,X^j,H^j)$ is uncolorable. Moreover, for each vertex $v \in V(G^j) \setminus \{v^*\}$ it holds $|X_v|=d_G(v)=d_{G^j}(v)$ (by \eqref{eq_degree}). As $(G^j,X^j,H^j)$ is uncolorable, it follows from Proposition~\ref{prop_main-proposition}(a) that $|X_{v^*}^j| \leq d_{G^j}(v^*)$ for $j \in \{1,2\}$. Since $X_{v^*}=X_1 \cup X_2 = X^1_{v^*} \cup X^2_{v^*}$ we conclude from \eqref{eq_degree} that
$$|X_{v^*}^1| + |X_{v^*}^2| \geq |X_{v^*}^1 \cup X_{v^*}^2| = |X_{v^*}| = d_{G}(v^*) = d_{G^1}(v^*) + d_{G^2}(v^*),$$
and, thus, $|X_{v^*}^j|=d_{G^j}(v^*)$ and $X_{v*}^1 \cap X_{v*}^2 = \varnothing$. Hence, $(G^j,X^j,H^j)$ is a degree-feasible configuration. Moreover, $H'=H^1 \cup H^2$ is a spanning subhypergraph of $H$ and $V(H^1) \cap V(H^2) = \varnothing$. So, $(G,X,H')$ is a degree-feasible configuration (by Proposition~\ref{prop_feas-config}(b)) and $(G,X,H')$ is obtained from two ismorphic copies of $(G^1,X^1,H^1)$ and $(G^2,X^2,H^2)$ by the merging operation. Clearly, $(G,X,H')$ is uncolorable. Otherwise, there would exist an independent transversal $T$ of $(X,H')$ and, by symmetry, $T$ would contain a vertex of $X_{v^*}^1$. But then, $T^1=T \cap V(H^1)$ would be an independent transversal of $(X^1,H^1)$, which is impossible. As $(G,X,H)$ is minimal uncolorable and as $H'$ is a spanning subhypergraph of $H$, this implies that $H=H'$ and $(G,X,H)$ is obtained from two isomorphic copies of $(G^1,X^1,H^1)$ and $(G^2,X^2,H^2)$ by the merging operation. By Proposition~\ref{prop_minimal-uncolorable}, both $(G^1,X^1,H^1)$ and $(G^2,X^2,H^2)$ are minimal uncolorable (and also degree-feasible). Applying the induction hypotheses leads to $(G^j,X^j,H^j)$ being constructible for $j \in \{1,2\}$, and so $(G,X,H)$ is constructible. Thus, the first case is complete.

\case{2}{$G$ is a block.} If $G$ contains any hyperedge $e$, then it follows from Proposition~\ref{prop_main-proposition}(c) that $G=<e>$ and $(G,X,H)$ is not colorable if and only if $(G,X,H)$ is an E-configuration. Thus, in the following we may assume that $G$ does not contain any hyperedges. We prove that $(G,X,H)$ is either a K-configuration or a C-configuration. This is done via a sequence of claims.

\begin{claim} \label{claim_g-v}
Let $v$ be an arbitrary vertex of $G$, let $x \in X_v$ be an arbitrary color, and let $(G',X',H')=(G,X,H)/(v,x)$. Then, there is a spanning subhypergraph $\tilde{H}$ of $H'$ such that $(G,X,\tilde{H})$ is minimal uncolorable. Moreover, $(G',X',\tilde{H})$ is constructible and so each block of $G'=G-v$ is a {\upshape DP}-brick.
\end{claim}
\begin{proof2}
Since $|G| \geq 2$ and $G$ is connected, $X_v \neq \varnothing$ (by \eqref{eq_degree}). Thus, $(G',X',H')=(G,X,H)/(v,x)$ is an uncolorable degree-feasible configuration (by Proposition~\ref{proposition_reduction}) and, therefore, there is a spanning subhypergraph $\tilde{H}$ of $H'$ such that $(G',X',\tilde{H})$ is minimal uncolorable. Then, the induction hypothesis implies that $(G',X',\tilde{H})$ is constructible, and, as $G'=G-v$, this particularly implies that each block of $G'$ is a DP-brick (since $G$ does not contain any hyperedge).
\end{proof2}

By a \textbf{multicycle} or \textbf{multipath} we mean a multigraph that can be obtained from a cycle, respectively a path, by replacing each edge $e$ of the cycle or path by a set of $t_e$ parallel edges, where $t_e \geq 1$. Given integers $s,t \geq 1$, we say that a graph $H$ is an $(s,t)$\textbf{-multicycle} if $H$ can be obtained from an even cycle $C$ by replacing each edge of a perfect matching of $C$ by a set of $s$ parallel edges and each other edge of $C$ by a set of $t$ parallel edges. Clearly, each $(s,t)$-multicycle is $r$-regular for $r=s+t$. Moreover, if $H$ is a regular multicycle, then either $H=tC_n$ for some integers $t \geq 1$ and $n \geq 3$, or $H$ is an $(s,t)$-multicycle for some integers $s,t \geq 1$.

\begin{claim} \label{claim_dp-brick}
The graph $G$ is a \upshape{DP}-brick.
\end{claim}
\begin{proof2}
Since $G$ is a block, Proposition~\ref{prop_main-proposition}(d) implies that $G$ is $r$-regular for some integer $r \geq 1$. For any vertex $v$ of $G$, each block of $G-v$ is a DP-brick (by Claim~\ref{claim_g-v}). Let $S$ denote the set of all vertices $v$ of $G$ such that $G-v$ is a block. Then, for every vertex $v \in S$, $G-v$ is a DP-brick and, therefore, regular. As $G$ is regular, too, for $v \in S$ there must be an integer $t_v \geq 1$ such that $\mu_G(u,v)=t_v$ for all $u \in V(G) \setminus \{v\}$. As a consequence, $S=V(G)$ and it clearly holds $t_v=t$ for all $v \in V(G)$. Thus, $G=tK_n$ with $n=|G|$.

It remains to consider the case that $S=\varnothing$. Let $v$ be an arbitrary vertex of $G$. Then, $G-v$ has at least two end-blocks and each block of $G-v$ is a DP-brick and therefore regular. Let $B$ be an arbitrary end-block of $G-v$. Then, $B$ is $t_B$-regular for some $t_B \geq 1$ and $B$ contains exactly one separating vertex $v_B$ of $G-v$. As $G$ is $r$-regular, there is an integer $s_B$ such that $\mu_G(u,v)=s_B$ for all vertices $u \in V(B) \setminus \{v_B\}$. As a consequence, $|B|=2$, since otherwise every vertex of $B-v_B$ belongs to $S$ and so $S \neq \varnothing$, which is impossible. Hence, $B=t_BK_2$, $r=t_B+s_B$, $V(B)=\{v',v_B\}$, and $N_G(v')=\{v,v_B\}$. Repeating the above argumentation with $v'$ instead of $v$ proves that $G$ is a multicycle. Since $G$ is regular, this implies that either $G=tC_n$ with $t \geq 1$ and $n \geq 3$, or $G$ is an $(s,t)$-multicycle with $s \neq t$. If $G=tC_n$, we are done. We prove that $G$ cannot be an $(s,t)$-multicycle by reductio ad absurdum. By symmetry, we may assume $1 \leq s < t$. By \eqref{eq_degree}, for each vertex $v$ we have $|X_v|=s+t$. Let $v \in V(G)$. Then, $G-v$ is a multipath and one end-block of $G-v$, say $B$, is a $tK_2$. Then, $B$ consists of two vertices $u$ and $w$ with $d_{G-v}(u)=t$ and $d_{G-v}(w)=s+t$. Let $x \in X_v$ be an arbitrary color and set $(G',X',H')=(G,X,H)/(v,x)$. Then, there is a spanning subgraph $\tilde{H}$ of $H'$ such that $(G',X',\tilde{H})$ is constructible (by Claim~\ref{claim_g-v}). Moreover, \eqref{eq_degree} together with Proposition~\ref{prop_block-structure} implies that $|X_{u}'|=t, |X_{w}'|=s+t$ and that there is a subset $X_w^1$ of $X_w'$ such that $|X^1_w|=t$ and $H^1=H[X_u' \cup X_w^1]$ is a $K_{t,t}$ with parts $X_u'$ and $X_w^1$. The graph $H^1$ is a subgraph of $H^2=H[X_u \cup X_w]$, and $H^2$ is a $t$-regular bipartite graph with parts $X_u$ and $X_w$ (by Proposition~\ref{prop_main-proposition}(d)). Since $|X_u|=|X_w|=s+t$ and $1 \leq s < t$, this is impossible and the claim is proven. 
\end{proof2}

By Claim~\ref{claim_dp-brick}, $G$ is either a $tK_n$ with $t \geq 1$ and $n \geq 2$, or $G=tC_n$ with $t \geq 1$ and $n \geq 4$. In order to complete the proof we show that in the first case, $(G,X,H)$ is a K-configuration, and, in the second case, $(G,X,H)$ is a C-configuration.

\begin{claim}
If $G=tK_n$ for integers $t \geq 1$ and $n \geq 2$, then $(G,X,H)$ is a {\upshape K}-configuration. 
\end{claim}

\begin{proof2}
Since $(G,X,H)$ is minimal uncolorable, for each vertex $v$ of $G$ and each pair $u,w$ of distinct vertices of $G$, it holds
\begin{itemize}
\item[(a)] $|X_v|=t(n-1)$ and $H[X_u \cup X_w]$ is a $t$-regular bipartite graph with parts $X_u$ and $X_w$
\end{itemize}
(by \eqref{eq_degree} and by Proposition~\ref{prop_main-proposition}(d)). If $n=2$, then $G$ has exactly two vertices, say $u$ and $w$, and $H[X_u \cup X_w]$ is a $K_{t,t}$ (by (a)), and so $(G,X,H)$ is a K-configuration as claimed. 

Now assume that $n \geq 3$. Let $v$ be an arbitrary vertex of $G$, and let $x \in X_v$ be an arbitrary color. Moreover, let $(G',X',H')=(G,X,H)/(v,x)$. Then, there is a spanning subgraph $\tilde{H}$ of $H'= H - (X_v \cup N_H(x))$ such that $(G',X',\tilde{H})$ is a constructible configuration (by Claim~\ref{claim_g-v}). As $G'=G-v=tK_{n-1}$, $(G',X',\tilde{H})$ is a K-configuration. Consequently, for every vertex $u \in V(G)$, there is a partition $(X_u^1,X_u^2,\ldots,X_u^{n-2})$ of $X_u'=X_u \setminus N_H(x)$ such that, for $i \in \{1,2,\ldots,n-2\}$, 
\begin{itemize}
\item[(b)] the graph $H^i=\tilde{H}[\bigcup_{u \in V(G')} X_u^i]$ is a $K_{(n-1,t)}$ whose partite sets are the sets $X_u^i$ with $u \in V(G')$, and $\tilde{H}=H^1 \cup H^2 \cup \ldots \cup H^{n-2}$.
\end{itemize}
For $u \in V(G')$ let $X_u^{n-1} = X_u \setminus X_u'$. Then, for every vertex $u \in V(G')$, $|X_u^{n-1}|=t$ and $(X_u^1,X_u^2, \ldots, X_u^{n-1})$ is a partition of $X_u$. Since $\tilde{H}$ is a spanning subgraph of $H'$, it follows from (a) and (b) that $H^i$ is an induced subgraph of $H$ (for $i \in \{1,2,\ldots,n-2\})$, and the graph 
$$H^{n-1}=H[\bigcup_{u \in V(G')}X_u^{n-1}]$$
is a $K_{(n-1,t)}$ whose partite sets are the sets $X_u^{n-1}$ with $u \in V(G')$. Moreover,
$$H - X_v = H^1 \cup H^2 \cup \ldots \cup H^{n-1}, \text{ and } N_H(x) = V(H^{n-1}).$$
Since the color $x \in X_v$ was chosen arbitrarily, this implies that for each $x \in X_v$ there is an index $i \in \{1,2,\ldots,n-1\}$ such that $N_H(x)=V(H^i)$, and, by (a) and (b), for each index $i \in \{1,2,\ldots,n-1\}$ there are exactly $t$ colors $x$ from $X_v$ such that $N_H(x)=V(H^i)$. As a consequence, there is a partition $(X_v^1,X_v^2, \ldots,X_v^{n-1})$ of $X_v$ such that $|X_v^i|=t$ and $N_H(x)=V(H^i)$ for $x \in X_v^i$ and for $i \in \{1,2,\ldots,n-1\}$. Hence, for $i \in \{1,2,\ldots,n\}$, the graph
$$H_i=H[\bigcup_{u \in V(G)}X_u^i]$$
is a $K_{(n,t)}$ whose partite sets are the sets $X_u^i$ with $u \in V(G)$, and, moreover, $H=H_1 \cup H_2 \cup \ldots \cup H_n$. Thus, $(G,X,H)$ is a K-configuration.
\end{proof2}

\begin{claim}
If $G=tC_n$ for integers $t \geq 1$ and $n \geq 4$, then $(G,X,H)$ is a {\upshape C}-configuration. 
\end{claim}

\begin{proof2}
Since $(G,X,H)$ is minimal uncolorable, for each vertex $v \in V(G)$ and each $2$-set $\{u,w\} \in A(G)$, it holds
\begin{itemize}
\item[(a)] $|X_v|=2t$ and $H[X_u \cup X_w]$ is a $t$-regular bipartite graph with parts $X_u$ and $X_w$
\end{itemize}
(by \eqref{eq_degree} and by Proposition~\ref{prop_main-proposition}(d)). Let $v$ be an arbitrary vertex of $G$, and let $x \in X_v$ be an arbitrary color. Moreover, let $(G',X',H')=(G,X,H)/(v,x)$. Then, there is a spanning subgraph $\tilde{H}$ of $H'=H-(X_v \cup N_H(x))$ such that $(G',X',\tilde{H})$ is a constructible configuration (by Claim~\ref{claim_g-v}). Since $G'=G-v=tP_{n-1}$, the vertices of $G'$ can be arranged in a sequence, say $v_1,v_2,\ldots,v_{n-1}$, such that two vertices are adjacent in $G'$ if and only if they are consecutive in the sequence. Note that $N_G(v)=\{v_1,v_{n-1}\}$ and each block of $G'$ is a $tK_2$. We claim that for each vertex $u$ of $G'$ there is a partition $(X_u^1,X_u^2)$ of $X_u$ such that the following conditions hold:
\begin{itemize}
\item[(b)] For every $i \in \{1,2\}$ and every $k \in \{1,2,\ldots,n-2\}$, the graph $H^i_k=H[X^i_{v_k}\cup X^i_{v_{k+1}}]$ is a $K_{t,t}$ whose partite sets are $X^i_{v_k}$ and $X^i_{v_{k+1}}$.
\item[(c)] The graph $H - X_v$ is the union of all graphs $H^i_k$ with $i \in \{1,2\}$ and $k \in \{1,2,\ldots,n-2\}$.
\item[(d)] If $n$ is even, then $N_H(x)=X^1_{v_1} \cup X^2_{v_{n-1}}$, or $N_H(x) = X^2_{v_1} \cup X^1_{v_{n-1}}$.
\item[(e)] If $n$ is odd, then $N_H(x)=X_{v_1}^1 \cup X_{v_{n-1}}^1$, or $N_H(x)=X^2_{v_1} \cup X^2_{v_{n-1}}$.
\end{itemize}
For $k \in \{1,2,\ldots,n-2\}$, the graph $B^k=G[\{v_k,v_{k+1}\}]$ is a block of $G'$. Clearly, $\mathcal{B}(G')=\{B^1,B^2,\ldots,B^{n-2}\}$ and the only end-blocks of $G'$ are $B^1$ and $B^{n-2}$. Since $(G',X',\tilde{H})$ is a constructible configuration and since each block of $G'$ is a $tK_2$, it follows from Proposition~\ref{prop_block-structure} that for each $k \in \{1,2,\ldots,n-2\}$ there is a uniquely determined cover $(\tilde{X}^k,\tilde{H}^k)$ of $B^k$ such that
\begin{itemize}
\item $\tilde{H}^k$ is a $K_{t,t}$ with parts $\tilde{X}^k_{v_k}$ and $\tilde{X}^k_{v_{k+1}}$,
\item $\tilde{H}$ is the disjoint union of the graphs $\tilde{H^1}, \tilde{H^2},\ldots,\tilde{H}^{n-2}$,
\item $X'_{v_1}=\tilde{X}^1_{v_1}$, $X_{v_k}'=\tilde{X}_{v_k}^{k-1} \cup \tilde{X}^k_{v_k} (k \in \{2,\ldots,n-2\})$, and $X_{v_{n-1}}'=\tilde{X}^{n-2}_{v_{n-1}}$.
\end{itemize}
Since $\{v_k,v_{k+1}\} \in A(G)$ for $k \in \{1,2,\ldots,n-2\}$, it follows from (a) that $\tilde{H}^k$ is an induced subgraph of $H$. Let $\tilde{X}^0_{v_1}=X_{v_1} \setminus X_{v_1}'$ and $\tilde{X}^{n-1}_{v_{n-1}}=X_{v_{n-1}} \setminus X'_{v_{n-1}}$.Then, both sets $\tilde{X}^0_{v_1}$ and $\tilde{X}^{n-1}_{v_{n-1}}$ have exactly $t$ elements, and $N_H(x)=\tilde{X}^0_{v_1} \cup \tilde{X}^{n-1}_{v_{n-1}}$. Furthermore, we conclude from (a) that, for $k \in \{1,2,\ldots,n-2\}$, 
\begin{itemize}
\item the graph $H[\tilde{X}^{k-1}_{v_k} \cup \tilde{X}^{k+1}_{v_{k+1}}]$ is a $K_{t,t}$ with parts $\tilde{X}^{k-1}_{v_k}$ and $\tilde{X}^{k+1}_{v_{k+1}}$.
\end{itemize}
If $n$ is even, we set
$$(X^1_{v_1},X^1_{v_2},\ldots,X^1_{v_{n-1}})=(\tilde{X}^1_{v_1}, \tilde{X}^1_{v_2}, \tilde{X}^3_{v_3}, \tilde{X}^3_{v_4}, \ldots, \tilde{X}^{n-3}_{v_{n-3}}, \tilde{X}^{n-3}_{v_{n-2}}, \tilde{X}^{n-1}_{v_{n-1}}),$$
and
$$(X^2_{v_1},X^2_{v_2},\ldots,X^2_{v_{n-1}})=(\tilde{X}^0_{v_1}, \tilde{X}^2_{v_2}, \tilde{X}^2_{v_3}, \tilde{X}^4_{v_4}, \tilde{X}^4_{v_5}, \ldots, \tilde{X}^{n-2}_{v_{n-2}}, \tilde{X}^{n-2}_{v_{n-1}}).$$
If $n$ is odd, let
$$(X^1_{v_1},X^1_{v_2},\ldots,X^1_{v_{n-1}})=(\tilde{X}^1_{v_1}, \tilde{X}^1_{v_2}, \tilde{X}^3_{v_3}, \tilde{X}^3_{v_4}, \ldots, \tilde{X}^{n-2}_{v_{n-2}}, \tilde{X}^{n-2}_{v_{n-1}}),$$
and
$$(X^2_{v_1},X^2_{v_2},\ldots,X^2_{v_{n-1}})=(\tilde{X}^0_{v_1}, \tilde{X}^2_{v_2}, \tilde{X}^2_{v_3},\ldots, \tilde{X}^{n-3}_{v_{n-3}}, \tilde{X}^{n-3}_{v_{n-2}}, \tilde{X}^{n-1}_{v_{n-1}}).$$
By using (a) and Proposition~\ref{prop_main-proposition}(b), it is easy to check that, for every vertex $u$ of $G'$, $(X_u^1,X_u^2)$ is a partition of $X_u$ such that the conditions (b), (c), (d), and (e) are satisfied. Since the color $x \in X_v$ was chosen arbitrarily, it follows from (a) and Proposition~\ref{prop_main-proposition}(b) that there is a partition $(X_v^1,X_v^2)$ of $X_v$ such that $|X_v^1|=|X_v^2|=t$ and the following conditions hold:
\begin{itemize}
\item If $n$ is even, then $N_H(x)=X_{v_1}^1 \cup X_{v_{n-1}}^2$ for all $x \in X_v^1$ and $N_H(x)=X_{v_1}^2 \cup X_{v_{n-1}}^1$ for all $x \in X_v^2$.
\item If $n$ is odd, then $N_H(x)=X_{v_1}^1 \cup X_{v_{n-1}}^1$ for all $x \in X_v^1$ and $N_H(x) = X_{v_1}^2 \cup X_{v_{n-1}}^2$ for all $x \in X_v^2$.
\end{itemize}
Clearly, this implies that $(G,X,H)$ is a C-configuration, and the claim is proven.
\end{proof2}
This settles Case 2. Hence, in both cases we showed that $(G,X,H)$ is a constructible configuration and the proof of the theorem is complete. 
\end{proof}

As mentioned earlier, Kim and Ozeki~\cite{KiOz17} characterized the 'bad' covers for non-DP-degree colorable graphs; many ideas of their proof are similar to ours. In our terminology, they proved the following:

\setcounter{theorem}{6}
\begin{theorem}[Kim and Ozeki, 19] \label{theorem_Kim-Ozeki}
Let $G$ be a graph and let $(G,X,H)$ be a degree-feasible configuration. Then, $G$ is not $(X,H)$-colorable if and only if for each block $B \in \mathcal{B}(G)$ there is a cover $(X^B,H^B)$ of $B$ such that the following statements hold.
\begin{itemize}
\item[\upshape (a)] For every block $B \in \mathcal{B}(G)$, the triple $(B,X^B,H^B)$ is a {\upshape K}-configuration, or a {\upshape C}-configuration.
\item[\upshape (b)] The graphs $H^B$ with $B \in \mathcal{B}(G)$ are pairwise disjoint and $H \supseteq \bigcup_{B \in \mathcal{B}(G)} H^B$.
\item[\upshape (c)] For each vertex $v \in V(G)$ it holds $X_v=\bigcup_{B \in \mathcal{B}(G), v \in V(B)}X_v^B$.
\end{itemize}
\end{theorem}

In particular, if $G$ itself is a block it follows from their theorem that $(G,X,H)$ is either a K-, or a C-configuration. Thus, by using their result we could omit Claim~3 and 4 in the proof of Theorem~\ref{theorem_main-result}. However, for the reader's convenience, we decided to display the entire proof so that the reader gets a complete presentation how both the characterization of the 'bad' blocks as well as the corresponding 'bad' covers in the hypergraph and the graph case can be done.

\subsection{A Brooks' Type Theorem for $\mathbf{\chi}_{\text{DP}}$}

The next two corollaries are direct consequences of Theorem~\ref{theorem_main-result} and Proposition~\ref{prop_block-structure}.

\begin{corollary}
Let $(G,X,H)$ be a degree-feasible configuration. If $(G,X,H)$ is minimal uncolorable, then for each block $B \in \mathcal{B}(G)$ there is a uniquely determined cover $(X^B,H^B)$ of $B$ such that the following statements hold.
\begin{itemize}
\item[\upshape (a)] For every block $B \in \mathcal{B}(G)$, the triple $(B,X^B,H^B)$ is a {\upshape K}-configuration, a {\upshape C}-configuration, or an {\upshape E}-configuration.
\item[\upshape (b)] The hypergraphs $H^B$ with $B \in \mathcal{B}(G)$ are pairwise disjoint and $H= \bigcup_{B \in \mathcal{B}(G)} H^B$.
\item[\upshape (c)] For each vertex $v \in V(G)$ it holds $X_v=\bigcup_{B \in \mathcal{B}(G), v \in V(B)}X_v^B$.
\end{itemize}
\end{corollary}

\begin{corollary}\label{cor_DP-hyper}
A connected hypergraph $G$ is not {\upshape DP}-degree-colorable if and only if each block of $G$ is a {\upshape DP}-hyperbrick.
\end{corollary}

To conclude this paper, we are now able to give a Brooks-type theorem for DP-colorings of hypergraphs. For graphs, the theorem was proven already by Bernshteyn, Kostochka, and Pron \cite{BeKoPro17}.

\begin{theorem}
Let $G$ be a connected hypergraph. Then, $\chi_{\rm{DP}}(G) \leq \Delta(G) + 1$ and equality holds if and only if $G$ is a {\upshape DP}-hyperbrick.
\end{theorem}
\begin{proof}
It follows from \eqref{eq_chi-inequalities} that $\chi_{\text{DP}}(G) \leq \Delta(G) + 1$ always holds. Moreover, it is obvious that any DP-hyperbrick $G$ satisfies $\chi_{\text{DP}}(G) = \Delta(G) + 1$, take any K-, C-, or E-configuration. Now assume that $\chi_{\text{DP}}(G) = \Delta(G) + 1$. Then, there is a cover $(X,H)$ of $G$ such that $|X_v| \geq \Delta(G)$ for all $v \in V(G)$ and $G$ is not $(X,H)$-colorable. Hence, $(G,X,H)$ is an uncolorable degree-feasible configuration and there is a spanning subhypergraph $H'$ of $H$ such that $(G,X,H')$ is minimal uncolorable. Then, $G$ is regular (by Proposition~\ref{prop_main-proposition}(a)) and each block of $G$ is a DP-hyperbrick (by Theorem~\ref{theorem_main-result}). As any DP-hyperbrick is regular, this implies that $G$ has only one block and, therefore, is a DP-hyperbrick. This completes the proof.
\end{proof}

\section{Concluding Remarks}
It is often of interest to determine the complexity of specific coloring problems. Clearly, a graph has chromatic number 2 if and only if it is bipartite. By König's Theorem this is equivalent to having no cycles of odd length, which can easily be checked in polynomial time. However, Lov\'asz \cite{Lov73} showed that for a fixed integer $k \geq 3$ it is an \NP-complete  decision problem to decide whether a graph admits a $k$-coloring.  Moreover, he proved that it is \NP-complete to decide wether a hypergraph is bipartite or not. This implies in particular that determining the chromatic number of a hypergraph is \NP-hard. Regarding list-colorings, Erd\H os, Rubin, and Taylor \cite{ErdRubTay79} and independently Vizing \cite{Vizing} showed that one can check in polynomial time if a graph admits an $L$-coloring provided that each vertex gets assigned a list of at most 2 colors. Furthermore, Erd\H os, Rubin, and Taylor \cite{ErdRubTay79} observed that, given a fixed integer $k \geq 3$, the problem if a graph is $k$-list colorable is $\Pi^p_2$-complete whereas $\Pi^p_2$ is a complexity class in the polynomial hierarchy containing both  \NP~and co\NP. Since DP-colorings are an extension of (list-)colorings of hypergraphs we conclude that, given a cover $(X,H)$ of a hypergraph $G$, it is  \NP -hard to decide if $G$ admits an $(X,H)$ coloring. Nevertheless, it might be an interesting topic to examine conditions under which a graph $G$ admits an $(X,H)$-coloring for some cover $(X,H)$. In order to get some ideas we recommend taking a look at a survey by Golovach, Johnson, Paulusma, and Song \cite{Gol17} that analyzes the complexity of coloring problems with respect to some forbidden subgraphs. Regarding list-colorings of (simple) hypergraphs with lists containing at least degree many colors it is easy to deduce a polynomial time algorithm from the proof of Kostochka, Stiebitz and Wirt \cite{KoStiWi96}  that, given a simple hypergraph $G$ and a list-assignment $L$ with $|L(v)| \geq d_G(v)$ for all $v \in V(H)$, either finds an $L$-coloring of $G$ or returns a 'bad' block. A similar algorithm for DP-degree colorability can be deduced from our proof.

\end{document}